\documentclass[11pt,reqno]{amsart}
\usepackage{amsmath, amssymb, amsthm}
\usepackage{url}
\usepackage{mathrsfs}
\usepackage[ansinew]{inputenc}
\usepackage[breaklinks]{hyperref}
\usepackage{fancyhdr}

\usepackage{caption}   
\theoremstyle{plain}
\theoremstyle{definition}
\newtheorem{definition}{Definition}[section]

\usepackage{multirow}

\allowdisplaybreaks

\renewcommand{\(}{\left\(}
\renewcommand{\)}{\right\)}
\renewcommand{\[}{\left\[}
\renewcommand{\]}{\right\]}
\newtheorem{remark}[]{Remark}
\numberwithin{equation}{section}
\theoremstyle{plain}
\newtheorem{theorem}{Theorem}[section]

\newcommand{\sm}{\left(\begin{smallmatrix}}
	\newcommand{\esm}{\end{smallmatrix}\right)}

\newtheorem{corollary}[theorem]{Corollary}

\newtheorem{example}[]{Example}

\makeatletter
\def\proof{\@ifnextchar[{\@oproof}{\@nproof}}
\def\@oproof[#1][#2]{\trivlist\item[\hskip\labelsep\textit{#2 \textbf{Proof of}\
		#1.}~]\ignorespaces}
\def\@nproof{\trivlist\item[\hskip\labelsep\textit{Proof.}~]\ignorespaces}

\makeatother

\usepackage{color}
\usepackage{amsmath}

\definecolor{blue}{rgb}{0,0,1}
\definecolor{red}{rgb}{1,0,0}
\definecolor{green}{rgb}{0,.6,.2}
\definecolor{purple}{rgb}{1,0,1}

\long\def\red#1\endred{{\color{red}#1}}
\long\def\blue#1\endblue{{\color{blue}#1}}
\long\def\purple#1\endpurple{{\color{purple}#1}}
\long\def\green#1\endgreen{{\color{green}#1}}

\begin{document}
	\title[On $k$-color mex-related partition functions]{On $k$-color mex-related partition functions}

	
	\author{Nargish Punia}
	\address{Department of Mathematics, Indian Institute of Technology, Roorkee-247667, Uttarakhand, India}
	\email{nargish@ma.iitr.ac.in}
	
	\maketitle
	\pagenumbering{arabic}
	\pagestyle{headings}
	\begin{abstract}
		In this paper, we investigate a set of $k$-color partitions and examine how its subsets relates to the theory of the minimal excludant function introduced by Andrews and Newman. Furthermore, we provide generalization of the partition functions introduced by Andrews and Bachraoui and also define several new partition functions within this framework.
		
	\end{abstract}
	
	\section{Introduction}\label{intro}
Fraenkel and Peled \cite{mex1} originally defined the minimal excludant function (mex-function) of a set $S$ of non-negative integers, which is defined as the smallest non-negative integer that is not contained in $S$:
	$$\text{mex}(S):=\text{min}(\mathbb{Z}_{\ge0}\backslash S).$$
	
 The mex-function appears extensively in combinatorial game theory, particulary within the Sprague-Grundy theory. It is used to analyze impartial games like Nim (see \cite{mex1}, \cite{mex2}). This function seems to be considered first by Grabner and Knopfmacher \cite{grap} in 2006, who call it the smallest gap. They defined the smallest gap of an integer partition as the smallest integer which is not a part of the partition. In 2011, Andrews \cite{andr} related the minimal excludant function to the Frobenius symbol representation of partitions. In this direction, Andrews and Newman \cite{mex} and Hopkins and Sellers \cite{hop} found the relation between the minimal excludant function and Dyson's statistic crank. Very recently, Andrews and Newman \cite{mex} defined the minimal excludant function for an integer partition $\lambda$ i.e $\text{mex}_{A,a}(\lambda)$ to be the smallest integer $\equiv a\ (\text{mod}\  A)$ which is not a part of $\lambda$. Then they defined $p_{A,a}(n)$
	(resp. $\bar{p}_{A,a}(n))$ as the number of partitions $\lambda$ of $n$ with
	$$\text{mex}_{A,a}(\lambda)\equiv a \ \text{mod}\ (2A)\ \ \ (\text{resp}.\ \text{mex}_{A,a}(\lambda)\equiv a+A \ \text{mod}\ (2A)) .$$
	Also Andrews and Newman in \cite{newman2} defined,
	$$\sigma\text{mex}(n):=\sum_{\lambda \vdash n}^{}{\text{mex}(\lambda)},$$
	where the sum is over all the partitions of $n$, and they obtained the following result:
	\begin{equation}\label{d}
		\sigma\text{mex}(n)=D_2(n),
	\end{equation}			
	where $D_2(n)$ denote the number of two color partitions of $n$ into distinct parts.
The generating function for identity \eqref{d} was established by Grabner and Knopfmacher \cite{grap} in their investigation of the smallest gap. For recent developments in this area, see {\cite{b1,b2,b3,b4}, to name a few.
	Very recently, Andrews and Bachraoui \cite{twocolor} extend this framework to include partitions with two colors (red and blue), applying specific parity and color-based constraints to the parts, for instance, allowing even parts to appear only in blue. In this paper, we remove the restrictions on even parts and then generalize it to $k$-color partitions, subject to specific counting rules. More precisely, we consider a sequence of partitions with parts into $k$-colors (say $color(1)$, $color(2)$,\dots$color(k)$) with some restrictions on counting of parts in $color(k)$ and establish connections to the theory of mex-function. We will also add some new partition function within the framework given by Andrews and Bachraoui \cite{twocolor} and also give the generalizations of the partition function discussed by them.
	 
	 Throughout the paper, we use the following standard $q$-series notations,
	 $$ (a; q)_0 = 1, \quad (a; q)_n = \prod_{j=0}^{n-1} (1 - aq^j), \quad (a; q)_\infty = \prod_{j=0}^{\infty} (1 - aq^j), \quad |q|<1,$$
	 $$(a_1, \dots, a_k; q)_\infty = \prod_{j=1}^{k} (a_j; q)_\infty. $$
	  Before proceeding further, we introduce certain partition classes.
	
\subsection{New Classes Of Partitions}\label{n}
\begin{definition}
	For $n\ge0$, let $\mathcal{C}_k(n)$ denote the set of partitions of $n$ into $k$-colors (say $color(1)$, 
	$color(2)$,\dots$color(k)$) and $|\mathcal{C}_k(n)|=C_k(n)$, where $|\mathcal{C}_k(n)|$ denote the cardinality of set $\mathcal{C}_k(n)$. Then clearly,
	$$\sum_{n=0}^{\infty}C_k(n)q^n=\frac{1}{(q;q)_{\infty}^k}.$$
\end{definition}
\begin{definition}\label{def}
	For $n\ge0$ and $k\ge1$, let ${C}_{k}^e(n)$ (resp. ${C}_{k}^o(n)$) denote the number of partitions in $\mathcal{C}_{k}(n)$ such that parts in $color(k)$ are even (resp. odd) in counting.
\end{definition}
Note that ${C}_{k}^e(n)$ (resp. ${C}_{k}^o(n)$) can be considered as a generalization of $P_e(n)$ (resp. $P_o(n)$) and clearly
$$C_1^e(n)=P_e(n)\  \text{and}\  C_1^o(n)=P_o(n),$$
where $P_e(n)$ (resp. $P_o(n)$) denote the number of partitions of $n$ having even (resp. odd) number of parts.

To establish the connection of these partition function with mex-theory, let 
 $A$ be a positive integer and $a$ be a non-negative integer. For any $\lambda \in \mathcal{C}_k(n)$, we define $\text{mex}_{A,a}(\lambda_k)$ be the smallest integer $\equiv a\ \text{mod}\ (A)$ which is not a
part in $color(k)$ of $\lambda$. Accordingly, let $\mathcal{P}_{A,a}(n,k)$ (resp. $\mathcal{\bar{P}}_{A,a}(n,k)$) denote the number of
partitions $\lambda$ in $\mathcal{C}_k(n)$ with
$$\text{mex}_{A,a}(\lambda_k)\equiv a \ \text{mod}\ (2A)\ \ (\text{resp.} \ \text{mex}_{A,a}(\lambda_k)\equiv a+A\ \text{mod}\ (2A)).$$
%
 Before stating our next results, we provide an illustrative example to clarify Definitions \ref{def}.
 \begin{example}	
 We examine the case where $n=3$, $
 	k=2$ and $A=2$, $a=1$. Let us denote part $\lambda$ in $color(1)$ by $\lambda_b$ and in $color(2)$ by $\lambda_r$ with order $\lambda_b\ge \lambda_r$.
 	\begin{center}
 		\renewcommand{\arraystretch}{1.2}
 		\begin{table}[htbp]
 			\centering
 			\renewcommand{\arraystretch}{1.2}
 			\begin{tabular}{c|cccc}
 				\hline
 				\textbf{$C_{2}(3)$} 
 				& \textbf{\hspace{0.9cm}$C_{2}^e(3)$} \qquad\qquad
 				& \textbf{\hspace{0.9cm}$C_{2}^o(3)$} \qquad\qquad
 				& \textbf{\hspace{0.9cm}$\mathcal{P}_{2,1}(3,2)$} \qquad\qquad
 				& \textbf{$\mathcal{\bar{P}}_{2,1}(3,2)$} \\
 				\hline
 				$3_r$           &  	$3_b$            &   	$3_r$         &	$3_r$        &	$1_r+2_b$           \\
 			$3_b$          &     	$1_r+2_r$        &    $1_r+2_b$       &	$3_b$ &    	$1_r+2_r$         \\
 			$1_r+2_b$          &$1_b+2_b$           &    	$1_b+2_r$       &   	$1_b+2_r$        &  	$1_r+1_r+1_r$             \\
 				$1_b+2_r$            &    	$1_b+1_r+1_r$         &    $1_r+1_r+1_r$         &	$1_b+2_b$    & 	$1_b+1_r+1_r$            \\
 				$1_r+2_r$       & 	$1_b+1_b+1_b$        &  	$1_b+1_b+1_r$          & 	$1_b+1_b+1_b$         & 	$1_b+1_b+1_r$            \\
 				$1_b+2_b$         &      &          &         &           \\
 				$1_r+1_r+1_r$      &            &   &        &           \\
 				$1_b+1_b+1_b$         &            &           &        &        \\
 					$1_b+1_r+1_r$         &            &           &        &        \\
 						$1_b+1_b+1_r$         &            &           &        &       \\
 				\hline
 			\end{tabular}
 				\caption{\ }
 			\label{t1}
 		\end{table}
 	\end{center}
Note that $C_{2}^e(3)=4=\mathcal{P}_{2,1}(3,2)$ and $C_{2}^o(3)=4=\mathcal{\bar{P}}_{2,1}(3,2)$. Our results, stated next, show that these identities hold for general $n$ and for any arbitrary number of colors $k$.
 \end{example}

\begin{theorem}\label{even}
For $n\ge 0$ we have,
$$C_k^e(n)=\mathcal{P}_{2,1}(n,k),$$
and
	$$C_{k}^o(n)=\mathcal{{\bar{P}}}_{2,1}(n,k).$$
\end{theorem}
This theorem yields a result by Andrews and Newman in \cite[Theorem $4$]{mex}, which can be stated as,
\begin{corollary}\label{c1}
	For $n\ge 0$ we have,
	$$\mathcal{P}_{2,1}(n,1)=C_1^e(n)=P_e(n),$$
	and
$$\mathcal{\bar{P}}_{2,1}(n,1)=C_1^o(n)=P_o(n).$$	
\end{corollary}
Now, our next result establishes a direct connection between partition function $C_k(n)$ and  mex-function. We will assume that $C_k(n)=0$, when $n$ is not a non negative integer.
\begin{theorem}\label{new}
	$C_k(n)-C_k(n-1)$ denote the excess number of partitions counted by $\mathcal{P}_{2,1}(n,k)$ over those partitions with $\text{mex}_{A,a}(\lambda_k)>1$, i.e.
	
	For any positive integer $n$, we have 
	$$C_k(n)-C_k(n-1)=\mathcal{P}_{2,1}(n,k)-\mathcal{P}_{2,1}^{>1}(n,k),$$ 
	where $\mathcal{P}_{2,1}^{>1}(n,k)$ denote the number of partitions counted by $\mathcal{P}_{2,1}(n,k)$ with $\text{mex}_{A,a}(\lambda_k)>1.$
\end{theorem}
This gives us the following corollary.
\begin{corollary}\label{new1}
	For $n\ge1$, there holds
	$$P_e(n)-P(n)+P(n-1)=\mathcal{P}_{2,1}^{>1}(n,1),$$
	where $P(n)$ denotes the total number of partitions of $n$.
\end{corollary}
This yields the subsequent result concerning the upper bound of $P(n)-P(n-1)$.
\begin{corollary}\label{new2}
	For $n\ge1$, there holds
	$$ P(n)-P(n-1)\le P_e(n).$$
\end{corollary}
%
Andrews and Bachraoui in \cite{twocolor} explicitly defined the following partition functions.
\begin{definition}
	For $n\ge0$, let $\mathcal{F}(n)$ denote the set of two color partitions of $n$ such that the even part may occur only in blue color and $F(n)=|\mathcal{F}(n)|.$ Now consider $\mathcal{H}(n)$ which denote the subset of $\mathcal{F}(n)$ wherein the parts of the same color do not repeat and $H(n)=|\mathcal{H}(n)|.$
\end{definition}
They also considered the following subsets of $\mathcal{F}(n)$ and $\mathcal{H}(n)$.
\begin{definition}
	 Let $F_0(n)$ (resp. $F_1(n)$) be the number of partitions in  $\mathcal{F}(n)$ in which number of odd parts in red color is even (resp. odd). Also, let $F_2(n)$ (resp. $F_3(n)$) be the number of partitions in  $\mathcal{F}(n)$ in which number of even parts is even (resp. odd).
	 
 Let $H_0(n)$ (resp. $H_1(n)$) be the number of partitions in  $\mathcal{H}(n)$ in which number of even parts is even (resp. odd). Also, let $H_2(n)$ (resp. $H_3(n)$) be the number of partitions in  $\mathcal{H}(n)$ in which the number of parts is even (resp. odd).	 
\end{definition}
	In this direction, we will extend these definitions for $k$-colors and also introduce some new kind of subsets.
	\begin{definition}
		For any non-negative integer $n$, let $\mathcal{N}_k(n)$ be the set of $k$-color partitions (say
		${color(1)},{color(2)},\dots,{color(k)}$) of $n$ such that the even parts may occur only in $k-1$ colors (say ${color(2)},{color(3)},\dots,{color(k)}$) and let $N_k(n)=|\mathcal{N}_k(n)|$, then we have
		$$\sum_{n=0}^{\infty}N_k(n)q^n=\frac{1}{(q;q^2)_{\infty}^k(q^2;q^2)_{\infty}^{k-1}}.$$
		Let $\mathcal{R}_k(n)$ be the subset of $\mathcal{N}_k(n)$, wherein the parts of the same color do not repeat and $R_k(n)=|\mathcal{R}_k(n)|$. Then, for $k\ge 2$
				\begin{align*}
			\sum_{n\ge0}^{}R_k(n)q^n&=(-q;q^2)_{\infty}^k(-q^2;q^2)_{\infty}^{k-1}\\
			&=(-q;q)_{\infty}^{k-1}(-q;q^2)_\infty.
		\end{align*}
	Now by using Euler's formula \cite[Corollary 1.2, p.~5]{andrewsbook}, i.e.
		$$(-q;q)_\infty=\frac{1}{(q;q^2)_\infty},$$
		we have
			\begin{align*}
			\sum_{n\ge0}^{}R_k(n)q^n&=\frac{(-q;q^2)_\infty}{(q;q^2)_{\infty}^{k-1}}\\
						&=\sum_{n\ge0}^{}\bar{O}_{k-1}(n)q^n,
		\end{align*}
			where $\bar{O}_{k}(n)$ denote the number of partitions of $n$ into odd parts only  of $k$ colors, where part of ${color(k)}$ may be overlined. 
	\end{definition}
\begin{definition}
	For a non-negative integer $n$, let $N_{0,k}(n)$ (resp. $N_{1,k}(n)$) denote the number of partitions of $n$ in $\mathcal{N}_k(n)$ such that parts in ${color(k)}$ are even (resp. odd) in counting.
	
	Now consider, $R_{0,k}(n)$ (resp. $R_{1,k}(n)$) denote the number of partitions of $n$ in $\mathcal{R}_k(n)$ such that parts in ${color(k)}$ are even (resp. odd) in counting. 
\end{definition}
	To bridge the concept of $N_{0,k}(n)$ and $N_{1,k}(n)$ with mex-theory, let us assume A is a positive integer and $a$ be a non-negative integer and $\lambda\in\mathcal{C}_k(n)$ such that the parts $\equiv 0\ (mod\ {A})$ may occur only in ${color(2)},{color(3)},\dots,{color(k)}$.
	Let $\text{mex}_{A,a}(\lambda,k)$ be the smallest integer $\equiv a\ (\text{mod}\ A)$ which is not a 
	part in $color(k)$ of $\lambda$. Accordingly, let ${P}_{A,a}(n,k)$ (resp. $\bar{P}_{A,a}(n,k)$ ) count the number of such
	partitions $\lambda$ of $n$ with
	$$\text{mex}_{A,a}(\lambda,k)\equiv a \ \text{mod}\ (2A)\ \ \ (\text{resp}.\ \text{mex}_{A,a}(\lambda,k)\equiv a+A \ \text{mod}\ (2A)) .$$
	Clearly, for $k\ge 3$, we have
	$$N_{0,k}(n)-N_{1,k}(n)=C_{k-2}(n),$$
	where $C_k(n)$, as defined earlier, denote the number of partitions of $n$ into $k$-colors.
	
	Interestingly, for $k=3$, there is a direct connection with partition function $P(n)$, i.e.
		$$N_{0,3}(n)-N_{1,3}(n)=P(n).$$
	
	To provide clarity of the definitions, we offer the following example.
	\begin{example}
		Take $n=3$, $
		k=3$ and $A=2$, $a=1$. Let us denote part $\lambda$ in $color(1)$ by $\lambda_b$, in $color(2)$ by $\lambda_r$ and in $color (3)$ by $\lambda_g$ with order $\lambda_b\ge \lambda_r\ge \lambda_g.$
		\begin{center}
			\renewcommand{\arraystretch}{1.2}
			\begin{table}[htbp]
				\centering
				\renewcommand{\arraystretch}{1.2}
				\begin{tabular}{c|cccc}
					\hline
					\textbf{$N_{3}(3)$} 
					& \textbf{\hspace{0.9cm}$N_{0,3}(3)$} \qquad\qquad
					& \textbf{\hspace{0.9cm}$N_{1,3}(3)$} \qquad\qquad
					& \textbf{\hspace{0.9cm}${P}_{2,1}(3,3)$} \qquad\qquad
					& \textbf{${\bar{P}}_{2,1}(3,3)$} \\
					\hline
					$3_b$           &  	$3_b$            &  	$3_g$       &	$3_b$        &	$1_g+2_g$           \\
					$3_r$          &     	$3_r$        &    $1_r+2_g$       &	$3_r$ &    	$1_g+2_r$       \\
					$3_g$          &     	$1_b+2_r$      &    	$1_b+2_g$        &	$3_g$ &    	$1_g+2_g$       \\
					$1_b+2_r$            &   $1_r+2_r$  	         &   	$1_g+2_r$          &		$1_b+2_r$     & 	  $1_g+1_g+1_g$           \\
					$1_b+2_g$            &    	$1_g+2_g$  	               & 	$1_r+1_r+1_g$     & $1_r+2_r$   & 	$1_b+1_r+1_g$          \\
					$1_r+2_g$       & 	$1_b+1_r+1_r$ 	     &    $1_b+1_r+1_g$       & 	$1_r+2_g$	      &   $1_b+1_g+1_g$            \\
					$1_r+2_r$       & 	$1_b+1_b+1_b$  		      &   			$1_b+1_b+1_g$        & 		$1_b +2_g$    &      	$1_r+1_r+1_g$         \\
					$1_g+2_r$       & $1_r+1_r+1_r$ 		      & $1_g+1_g+1_g$         &     $1_b+1_b+1_b$      &     	$1_r+1_g+1_g$          \\
					$1_g+2_g$       & 	$1_b+1_b+1_r$      &        & $1_b+1_r+1_r$    &            \\
					$1_b+1_b+1_b$         &  $1_b+1_g+1_g$             &           &  	$1_b+1_b+1_r$   	             &        \\
					$1_r+1_r+1_r$      &  	$1_r+1_g+1_g$    	       &   &   	$1_r+1_r+1_r$     &           \\
					$1_g+1_g+1_g$         &   		           &           &   	       &        \\
					$1_b+1_b+1_r$         &            &           &         &        \\
					$1_b+1_b+1_g$         &            &           &          &       \\
					$1_b+1_r+1_g$         &          &           &        &       \\
					$1_b+1_g+1_g$         &            &           &        &       \\
						$1_b+1_r+1_r$         &            &           &        &       \\
					$1_r+1_r+1_g$         &            &           &        &       \\
					$1_r+1_g+1_g$         &            &           &        &       \\
					\hline
				\end{tabular}
				\caption{\ }
				\label{t2}
			\end{table}
		\end{center}
			The Table \ref{t2} shows that $N_{0,3}(3)=11={P}_{2,1}(3,3)$ and $N_{1,3}(3)=8={\bar{P}}_{2,1}(3,3)$ and $N_{0,3}(3)-N_{1,3}(3)=3=P(3).$
\end{example}	

These identities are not merely a coincidence, but our next results shows that these holds for general $n$ and for any arbitrary number of colors $k$.
	\begin{theorem}\label{main}
			For any $n\ge 0$, we have
		$$N_{0,k}(n)={P}_{2,1}(n,k),$$
		and
			$$N_{1,k}(n)=\bar{{P}}_{2,1}(n,k). $$
	\end{theorem}
%
	Our next goal is to prove the following theorem which connect $N_{0,k}(n)$ and $N_{1,k}(n)$ with overpartitions. An overpartition of $n$ is a non-increasing sequence of natural numbers whose sum
	is $n$ in which the first occurrence (equivalently, the final occurrence) of a number
	may be overlined (see \cite{overp}).
	
	We will assume that $\bar{C}_k(n)=0$ for $k\le 0.$
		\begin{theorem}\label{over} For $k\in \mathbb{N}$, we have
		$$N_{0,k}(n)=\frac{\bar{{C}}_{k-1}+{C}_{k-2}(n)}{2},\ \ \ N_{1,k}(n)=\frac{\bar{{C}}_{k-1}-{C}_{k-2}(n)}{2},$$
		where $\bar{{C}}_{k}(n)$ denote the number of partitions of $n$ into $k$ colors where parts in ${color(k)}$ may be overlined.
	\end{theorem}
	Now setting $k=3$ in the expression for $\mathcal{R}_k(n)$ reveals that its coefficients satisfy the following congruences modulo 2.
	\begin{theorem}\label{O}
We have,
			$$
		\mathcal{R}_3(n)=	\bar{O}_{2}(n) \equiv P(d,n)\equiv
			\begin{cases}
				0\pmod 2 & when\  n\neq\frac{m(3m-1)}{2}\\
				1\pmod 2 & when\  n=\frac{m(3m-1)}{2},
			\end{cases}
$$
where $P(d,n)$ denote the total number of partition of $n$ into distinct parts only.
	\end{theorem}
	Next we have the following theorem.
	\begin{theorem}\label{R1}
		 For any nonnegative integer $n$, we have 
		\begin{enumerate}
			\item $R_{0,3}(n) = \begin{cases} \frac{\bar{O}_{2}(n) }{2} + \frac{(-1)^{\frac{m(3m\mp1)}{2}}}{2} & \text{if } n=\frac{m(3m\pm1)}{2} \\
				0& \text{otherwise,} \end{cases}$
			\item $R_{1,3}(n) = \begin{cases} \frac{\bar{O}_{2}(n) }{2} - \frac{(-1)^{\frac{m(3m\mp1)}{2}}}{2} & \text{if } n=\frac{m(3m\pm1)}{2} \\
				0& \text{otherwise.} \end{cases}$
		\end{enumerate}
	\end{theorem}
	This theorem gives us following corollary.
		\begin{corollary}\label{c3}
		$$R_{0,3}(n)-R_{1,3}(n) = \begin{cases}  (-1)^{\frac{m(3m\mp1)}{2}} & \text{if } n=\frac{m(3m\pm1)}{2} \\
			0& \text{otherwise.} \end{cases}$$
	\end{corollary}
	\begin{remark}
	We note that the case $k=4$ is quite interesting in the following way:
	$$R_{0,4}(n)-R_{1,4}(n)\equiv(q;q)_{\infty}^2 (\text{mod}\ 2).$$
Hence, under $(\text{mod}\ 2)$ this difference is directly connected with square of Euler's series, which is more deeply studied by Glaisher in \cite{glash}.
\end{remark}
\subsection{Generalizations Of Partition Functions Of Andrews-Bachraoui}\label{g}
\begin{definition}
	Let $N_{2,k}(n)$ (resp. $N_{3,k}(n)$) denote number of partitions in $\mathcal{N}_k(n)$ such that odd parts in color $k$ are even (resp. odd) in counting. For example, $N_{2,2}(3)=4$ with admissible partitions, $3_b,\ 1_b+2_r,\ 1_b+1_r+1_r$, $1_b+1_b+1_b$ and $N_{3,k}(3)=4$ with admissible partitions, $3_r,\ 1_r+2_r,\ 1_r+1_r+1_r,\ 1_b+1_b+1_r$.

Also let $N_{4,k}(n)$ (resp. $N_{5,k}(n)$) denote number of partitions in $\mathcal{N}_k(n)$ such that number of even parts in $color(k)$ are even (resp. odd).
\end{definition}
	\begin{theorem}\label{gen}
			$$	\sum_{n=0}^{\infty}N_{2,k}(n)q^n=\frac{(q^{16},-q^6,-q^{10};q^{16})_\infty}{(q;q^2)_{\infty}^{k-1}(q^2;q^2)_{\infty}^k}.$$
	\end{theorem}
\begin{corollary}\label{gc1}
	For any nonnegative integer $n$, we have that $N_{2,k}(n)$ equals the number
	of partitions of $n$ in $k$ colors wherein the odd parts and the parts
	$\equiv 0\  (\text{mod}\  16) $  may appear only in $k-1$ colors and the parts $\equiv 6,10\  (\text{mod}\  16) $ in $color(k)$ may be overlined.
\end{corollary}
		Theorem \ref{gen} provides corollaries established by Andrews and Bachraoui \cite[Theorem 1]{twocolor}:
	\begin{corollary}\label{p}
	We have
		$$N_{2,2}(n)=F_0(n).$$
	\end{corollary}
	\begin{corollary}\label{nnn}
			For any nonnegative integer $n$, we have that $F_0(n)$ equals the number
		of partitions of $n$ into two colors (red and blue) wherein the odd parts and the parts
		$\equiv 0\  (\text{mod}\  16) $  may appear only in red color and the parts $\equiv 6,10\  (\text{mod}\  16) $ in blue color may be overlined.
	\end{corollary}
	\begin{theorem}\label{gen2}
		$$	\sum_{n=0}^{\infty}N_{3,k}(n)q^n=q\frac{(q^{16},-q^2,-q^{14};q^{16})_\infty}{(q;q^2)_{\infty}^{k-1}(q^2;q^2)_{\infty}^k}.$$
	\end{theorem}
	The theorem above immediately gives us the following important corollary.
	\begin{corollary}\label{gc2}
	For any nonnegative integer $n$, we have that $N_{2,k}(n)$ equals the number
of partitions of $n-1$ in $k$ colors wherein the odd parts and the parts
$\equiv 0\  (\text{mod}\  16) $  may appear only in $k-1$ colors and the parts $\equiv 2,14\  (\text{mod}\  16) $ in $color(k)$ may be overlined.
\end{corollary}
	Note that for $k=2$, we present corollaries derived from the work of Andrews and Bachraoui \cite[Theorem 2]{twocolor}.
	\begin{corollary}\label{gc3}
		We have
				$$N_{3,2}(n)=F_1(n).$$
	\end{corollary}
	\begin{corollary}\label{nnnn}
			For any nonnegative integer $n$, we have that $F_1(n)$ equals the number
		of partitions of $n-1$ into two colors (red and blue) wherein the odd parts and the parts
		$\equiv 0\  (\text{mod}\  16) $  may appear only in red color and the parts $\equiv 2,14\  (\text{mod}\  16) $ in blue color may be overlined.
	\end{corollary}
	To establish the connection of these partition functions $N_{4,k}(n)$ and $N_{5,k}(n)$ with mex-theory,
assume that A is even, $a$ be a non-negative integer and $\lambda\in\mathcal{C}_k(n)$ such that the parts $\equiv 0\ (\text{mod}\ \frac{A}{2})$ may occur only in ${color(2)},{color(3)},
\dots,{color(k)}$.
	Let $\text{Mex}_{A,a}(\lambda,k)$ be the smallest integer $\equiv a\ (\text{mod}\ A)$ which is not
		part of $\lambda$. Accordingly, let $p_{A,a}(n,k)$ (resp. $\bar{p}_{A,a}(n,k)$) count the number of
		such partitions $\lambda$ of $n$ with
		
		$\text{Mex}_{A,a}(\lambda,k)\equiv a \ \text{mod}\ (2A)\ \ \ (\text{resp}.\ \text{Mex}_{A,a}(\lambda,k)\equiv a+A \ \text{mod}\ (2A))$.
		
	Now we have the following results for $N_{4,k}(n)$  and $N_{5,k}(n)$.
		\begin{theorem}\label{gen3}
		For any $n\ge 0$ we have,
		$$N_{4,k}(n)={p}_{4,2}(n,k),$$
		and
			$$N_{5,k}(n)=\bar{p}_{4,2}(n,k).$$
	\end{theorem}
	Theorem \ref{gen3} gives us results of Andrews and Bachraoui \cite[Theorem 3, Theorem 4]{twocolor}.
		\begin{corollary}\label{gc4}
	We have
		$$N_{4,2}(n)=F_2(n),$$
		and
			$$N_{5,2}(n)=F_3(n).$$
	\end{corollary}
We illustrate these results by considering the following example.
\begin{example}
We take $n=3$, $
k=3$ and $A=4$, $a=2$.
\begin{center}
	\renewcommand{\arraystretch}{1.2}
	\begin{table}[htbp]
		\centering
		\renewcommand{\arraystretch}{1.2}
		\begin{tabular}{c|cccc}
			\hline
			\textbf{$N_{3}(3)$} 
			& \textbf{\hspace{0.9cm}$N_{4,3}(3)$} \qquad\qquad
			& \textbf{\hspace{0.9cm}$N_{5,3}(3)$} \qquad\qquad
			& \textbf{\hspace{0.9cm}${p}_{4,2}(3,3)$} \qquad\qquad
			& \textbf{${\bar{p}}_{4,2}(3,3)$} \\
			\hline
			$3_b$           &  	$3_b$            &     $1_b+2_g$        &	$3_b$        &	$1_b+2_g$           \\
			$3_r$          &     	$3_r$        &    $1_r+2_g$       &	$3_r$ &    	$1_r+2_g$       \\
				$3_g$          &     	$3_g$        &    $1_g+2_g$       &	$3_g$ &    	$1_g+2_g$       \\
			$1_b+2_r$            &    		$1_b+2_r$          &         &		$1_b+2_r$     & 	           \\
				$1_b+2_g$            &    		$1_r+2_r$                &   & $1_r+2_r$   &         \\
			$1_r+2_g$       & 		$1_g+2_r$       &        & 	$1_g+2_r$	      &            \\
			$1_r+2_r$       & 		$1_g+1_g+1_g$       &        & 		$1_b+1_b+1_b$         &            \\
			$1_g+2_r$       & 	$1_b+1_b+1_b$  	      &        &     $1_b+1_b+1_b$      &            \\
				$1_r+2_r$       & 	$1_r+1_r+1_r$ 	     &        &  $1_g+1_g+1_g$   &            \\
				$1_b+1_b+1_b$         &   $1_b+1_b+1_r$          &           &  	$1_b+1_b+1_r$   	             &        \\
			$1_r+1_r+1_r$      &  		$1_b+1_b+1_g$          &   &   $1_b+1_b+1_g$    &           \\
				$1_g+1_g+1_g$         &   	   $1_b+1_r+1_g$             &           &   	$1_b+1_r+1_g$        &        \\
			$1_b+1_b+1_r$         &     	$1_r+1_r+1_g$        &           &  	$1_r+1_r+1_g$         &        \\
			$1_b+1_b+1_g$         &      $1_r+1_g+1_g$ 	         &           & $1_r+1_g+1_g$           &       \\
				$1_b+1_r+1_g$         &    $1_b+1_g+1_g$        &           &  $1_b+1_g+1_g$          &       \\
					$1_b+1_g+1_g$         & $1_b+1_r+1_r$            &           & $1_b+1_r+1_r$        &       \\
						$1_b+1_r+1_r$         &            &           &        &       \\
				$1_r+1_r+1_g$         &            &           &        &       \\
					$1_r+1_g+1_g$         &            &           &        &       \\
			\hline
		\end{tabular}
		\caption{\ }
		\label{t3}
	\end{table}
\end{center}
The table \ref{t3} shows that $N_{4,3}(3)=16={p}_{4,2}(3,3)$ and $N_{5,3}(3)=3={\bar{p}}_{4,2}(3,3)$.
\end{example}	
This paper is organized as follows. Section \ref{proof} presents the proofs for all results stated in section \ref{n}. While the proofs for results in section \ref{g} are provided in section \ref{gene}.
	\section{Proofs for new classes of partitions}\label{proof}
		\begin{proof}[\textbf{Theorem \textup{\ref{even}}}][]
				It is clear that,
			\begin{align}\label{diff1}
				\sum_{n=0}^{\infty}\left(C_{k}^e(n)-C_{k}^o(n)\right)q^n=\frac{1}{(q;q)_{\infty}^{k-1}(-q;q)_\infty},
			\end{align}
			and
			\begin{align}\label{sum1}
				\sum_{n=0}^{\infty}\left(C_{k}^e(n)+C_{k}^o(n)\right)q^n=\frac{1}{(q;q)_{\infty}^{k}}.
			\end{align}
			Now adding \eqref{diff1} and \eqref{sum1}, we have
			\begin{align}
				\sum_{n=0}^{\infty}C_{k}^e(n)q^n&=\frac{1}{2(q;q)_{\infty}^{k-1}}\left[\frac{1}{(-q;q)_\infty}+\frac{1}{(q;q)_\infty}\right]\nonumber\\
				&=\frac{1}{2(q;q)_{\infty}^{k-1}(q;q)_\infty}\left[\frac{(q;q)_\infty}{(-q;q)_\infty}+1\right].\label{j}
			\end{align}
			Now using Jacobi's triple product identity \cite[p. 12]{hyper}
			\begin{align}
				\sum_{n=-\infty}^{\infty}x^nq^{\frac{n(n+1)}{2}}=(q,-xq,-1/x;q)_\infty \label{jacobi}.
			\end{align}
			From \eqref{j}, we get
			\begin{align*}
				\sum_{n=0}^{\infty}C_{k}^e(n)q^n&=\frac{1}{2(q;q)_{\infty}^{k}}\left[\sum_{n=-\infty}^{\infty}(-1)^nq^{n^2}+1\right]\\
				&=\frac{1}{(q;q)_{\infty}^{k}}\sum_{n=0}^{\infty}(-1)^nq^{n^2}.
			\end{align*}
			On separating the domain of summation into even and odd indices, we find that
			$$\sum_{n=0}^{\infty}C_{k}^e(n)q^n		=\frac{1}{(q;q)_{\infty}^{k}}\sum_{n=0}^{\infty}q^{4n^2}(1-q^{4n+1}).$$
			Now using the basic fact that, $1+3+5+\dots+(2n-1)=n^2,$
			$$\sum_{n=0}^{\infty}C_{k}^e(n)q^n	=\frac{1}{(q;q)_{\infty}^{k-1}(q;q)_\infty}\sum_{n=0}^{\infty}q^{1+3+5+\dots+(2(2n)-1)}(1-q^{4n+1})$$
			$$=\frac{1}{(q;q)_{\infty}^{k-1}}\sum_{n=0}^{\infty}\frac{q^{1+3+5+\dots+(4n-1)}}{\prod\limits_{\substack{j=1 \\ j \neq 4n+1}}^{\infty}(1-q^{j})}.$$
			Now subtracting \eqref{diff1} from \eqref{sum1}, we find that
			\begin{align*}
				\sum_{n=0}^{\infty}C_{k}^o(n)q^n&=\frac{1}{2(q;q)_{\infty}^{k-1}}\left[\frac{1}{(q;q)_\infty}-\frac{1}{(-q;q)_\infty}\right]\\
				&=\frac{1}{2(q;q)_{\infty}^{k-1}(q;q)_\infty}\left[-\frac{(q;q)_\infty}{(-q;q)_\infty}+1\right].
			\end{align*}
			Now using the fact \eqref{jacobi}, we obtain
			\begin{align*}
				\sum_{n=0}^{\infty}C_{k}^o(n)q^n&=\frac{1}{2(q;q)_{\infty}^{k}}\left[1-\sum_{n=-\infty}^{\infty}(-1)^nq^{n^2}\right]\\
				&=-\frac{1}{(q;q)_{\infty}^{k}}\sum_{n=1}^{\infty}(-1)^nq^{n^2}\\
				&=\frac{1}{(q;q)_{\infty}^{k}}\sum_{n=0}^{\infty}(-1)^nq^{n^2+1+2n}\\
				&=\frac{1}{(q;q)_{\infty}^{k}}\sum_{n=0}^{\infty}q^{4n^2+4n+1}(1-q^{4n+3})\\
				&=\frac{1}{(q;q)_{\infty}^{k-1}}\sum_{n=0}^{\infty}\frac{q^{1+3+5+\dots+(2(2n+1)-1)}}{\prod\limits_{\substack{j=1 \\ j \neq 4n+3}}^{\infty}(1-q^{j})}.
			\end{align*}
			Hence the result.
		\end{proof}
			\begin{proof}[\textbf{Corollary \textup{\ref{c1}}}][]
			Put $k=1$ in \eqref{j}, we obtain,
					\begin{align}
						\sum_{n=0}^{\infty}C_{1}^e(n)q^n &=\frac{1}{2}\left[\frac{1}{(-q;q)_\infty}+\frac{1}{(q;q)_\infty}\right]\nonumber\\
						&=\frac{1}{2}\left[\sum_{n=0}^{\infty}\frac{(-1)^nq^n}{(q;q)_n}+\sum_{n=0}^{\infty}\frac{q^n}{(q;q)_n}\right]\label{id}\\
						&=\frac{1}{2}\sum_{n=0}^{\infty}\frac{(1+(-1)^n)q^n}{(q;q)_n}\nonumber\\
						&=\sum_{n=0}^{\infty}\frac{q^{2n}}{(q;q)_{2n}}=\sum_{n=0}^{\infty}P_e(n)q^n,\nonumber			
				\end{align}
				where in \eqref{id} we use \cite[p.~19, Cor.~2.2]{andrewsbook},
					for $|t|<1$, $|q|<1$,
				$$1+\sum_{n=1}^{\infty}\frac{t^n}{(q;q)_n}=\prod_{n=0}^{\infty}(1-tq^n)^{-1}.$$
			\begin{proof}[\textbf{Theorem \textup{\ref{new}}}][]
				Note that,
				\begin{align*}
			\sum_{n=0}^{\infty}(C_k(n)-C_k(n-1))q^n&=\frac{1}{(q;q)_\infty^k}-\frac{q}{(q;q)_\infty^k}\\
			&=\frac{(1-q)}{(q;q)_\infty^k}=\frac{1}{(q;q)_\infty^k}\left[\sum_{n=0}^{\infty}q^{4n^2}(1-q^{4n+1})-\sum_{n=1}^{\infty}q^{4n^2}(1-q^{4n+1})\right]\\
			&=\frac{1}{(q;q)_\infty^{k}}\left[\sum_{n=0}^{\infty}q^{1+3+5+\dots+(2(2n)-1)}(1-q^{4n+1})\right.\\
				&\ \ \ \ \ \ \ \ \ \ \ \ \ \ \ \left.-\sum_{n=1}^{\infty}q^{1+3+5+\dots+(2(2n)-1)}(1-q^{4n+1})\right]\\
					&=\frac{1}{(q;q)_\infty^{k-1}}\left[\sum_{n=0}^{\infty}\frac{q^{1+3+5+\dots+(2(2n)-1)}}{\prod\limits_{\substack{j=1 \\ j \neq 4n+1}}^{\infty}(1-q^{j})}-\sum_{n=1}^{\infty}\frac{q^{1+3+5+\dots+(2(2n)-1)}}{\prod\limits_{\substack{j=1 \\ j \neq 4n+1}}^{\infty}(1-q^{j})}\right]\\
					&=\sum_{n=0}^{\infty}\mathcal{P}_{2,1}(n,k)-\mathcal{P}_{2,1}^{>1}(n,k)q^n,
				\end{align*}
				on compairing coefficients of $q^n$ on both sides, we get the desired result.
			\end{proof}
			
					\begin{proof}[\textbf{Corollary \textup{\ref{new1}}}][]
						Putting $k=1$ in Theorem \ref{new}, we obtain
						$$P(n)-P(n-1)=\mathcal{P}_{2,1}(n,1)-\mathcal{P}_{2,1}^{>1}(n,1).$$
						Now using Corollary \ref{c1}, i.e.
						$$\mathcal{P}_{2,1}(n,1)=P_e(n).$$
			This implies,
							$$P_e(n)-P(n)+P(n-1)=\mathcal{P}_{2,1}^{>1}(n,1).$$
						\end{proof}
							\begin{proof}[\textbf{Corollary \textup{\ref{new2}}}][]
								This immediately follows from Corollary \ref{new1}.
				\end{proof}
					\end{proof}
	\begin{proof}[\textbf{Theorem \textup{\ref{main}}}][]
			It is clear that,
		\begin{align}\label{diff}
			\sum_{n=0}^{\infty}\left(N_{0,k}(n)-N_{1,k}(n)\right)q^n=\frac{1}{(q;q^2)_{\infty}^{k-1}(q^2;q^2)_{\infty}^{k-2}(-q;q)_\infty},
		\end{align}
		and
		\begin{align}\label{sum}
			\sum_{n=0}^{\infty}\left(N_{0,k}(n)+N_{1,k}(n)\right)q^n=\frac{1}{(q;q^2)_{\infty}^{k-1}(q^2;q^2)_{\infty}^{k-2}(q;q)_\infty}.
		\end{align}
		Now adding \eqref{diff} and \eqref{sum}, we find that
		\begin{align*}
			\sum_{n=0}^{\infty}N_{0,k}(n)q^n&=\frac{1}{2(q;q^2)_{\infty}^{k-1}(q^2;q^2)_{\infty}^{k-2}}\left[\frac{1}{(-q;q)_\infty}+\frac{1}{(q;q)_\infty}\right]\\
			&=\frac{1}{2(q;q^2)_{\infty}^{k-1}(q^2;q^2)_{\infty}^{k-2}(q;q)_\infty}\left[\frac{(q;q)_\infty}{(-q;q)_\infty}+1\right].
		\end{align*}
		 Jacobi's triple product identity \eqref{jacobi} implies that
		\begin{align*}
				\sum_{n=0}^{\infty}N_{0,k}(n)q^n&=\frac{1}{2(q;q^2)_{\infty}^{k-1}(q^2;q^2)_{\infty}^{k-1}(q;q^2)_\infty}\left[\sum_{n=-\infty}^{\infty}(-1)^nq^{n^2}+1\right]\\
				&=\frac{1}{(q;q^2)_{\infty}^{k-1}(q^2;q^2)_{\infty}^{k-1}(q;q^2)_\infty}\sum_{n=0}^{\infty}(-1)^nq^{n^2}\\
				&=\frac{1}{(q;q^2)_{\infty}^{k-1}(q^2;q^2)_{\infty}^{k-1}(q;q^2)_\infty}\sum_{n=0}^{\infty}q^{4n^2}(1-q^{4n+1})\\
				&=\frac{1}{(q;q^2)_{\infty}^{k-1}(q^2;q^2)_{\infty}^{k-2}}\sum_{n=0}^{\infty}\frac{q^{1+3+5+\dots+4n-1}}{\prod\limits_{\substack{j=1 \\ j \neq 4n+1}}^{\infty}(1-q^{j})}.
		\end{align*}
		This proves the first part, we now prove the second part. To that end, subtracting \eqref{diff} and \eqref{sum}, we have
			\begin{align*}
				\sum_{n=0}^{\infty}N_{1,k}(n)q^n&=\frac{1}{2(q;q^2)_{\infty}^{k-1}(q^2;q^2)_{\infty}^{k-2}}\left[\frac{1}{(q;q)_\infty}-\frac{1}{(-q;q)_\infty}\right]\\
			&=	\frac{1}{2(q;q^2)_{\infty}^{k-1}(q^2;q^2)_{\infty}^{k-2}(q;q)_\infty}\left[1-\frac{(q;q)_\infty}{(-q;q)_\infty}\right].
			\end{align*}
		Now, using \eqref{jacobi} with $x\rightarrow-\frac{1}{q}$ and $q\rightarrow q^2$, we have
			\begin{align*}
				\sum_{n=0}^{\infty}N_{1,k}(n)q^n	&=\frac{1}{2(q;q^2)_{\infty}^{k-1}(q^2;q^2)_{\infty}^{k-2}(q;q)_\infty}\left[1-\sum_{n=-\infty}^{\infty}(-1)^{n}q^{n^2}\right]\\
				&=\frac{1}{(q;q^2)_{\infty}^{k-1}(q^2;q^2)_{\infty}^{k-1}(q;q^2)_\infty}\sum_{n=1}^{\infty}(-1)^{n+1}q^{n^2}\\
				&=\frac{1}{(q;q^2)_{\infty}^{k-1}(q^2;q^2)_{\infty}^{k-1}(q;q^2)_\infty}\sum_{n=0}^{\infty}(-1)^{n}q^{(n+1)^2}\\
				&=\frac{1}{(q;q^2)_{\infty}^{k-1}(q^2;q^2)_{\infty}^{k-1}(q;q^2)_\infty}\sum_{n=0}^{\infty}q^{4n^2+4n+1}(1-q^{4n+3})\\
				&=\frac{1}{(q;q^2)_{\infty}^{k-1}(q^2;q^2)_{\infty}^{k-2}}\sum_{n=0}^{\infty}\frac{q^{1+3+5+\dots+4n+1}}{\prod\limits_{\substack{j=1 \\ j \neq 4n+3}}^{\infty}(1-q^{j})}\\
					&=\sum_{n=0}^{\infty}\bar{P}_{4,2}(n,k)q^n. 
			\end{align*}
			This completes the proof.
			\end{proof}
				\begin{proof}[\textbf{Theorem \textup{\ref{over}}}][]
				We already observed that,
				\begin{equation}\label{27}
					N_{0,k}(n)-N_{1,k}(n)=C_{k-2}(n).
				\end{equation}
					Now,
				\begin{align}
					\sum_{n=0}^{\infty}\left(N_{0,k}(n)+N_{1,k}(n)\right)q^n&=\frac{1}{(q;q^2)_{\infty}^{k-1}(q^2;q^2)_{\infty}^{k-2}(q;q)_\infty}\nonumber\\
					&=\frac{(-q;q)_\infty}{(q;q^2)_{\infty}^{k-1}(q^2;q^2)_{\infty}^{k-1}}\nonumber\\
					&=\frac{(-q;q)_\infty}{(q;q)_{\infty}^{k-1}}\nonumber\\
					&=\sum_{n=0}^{\infty}\bar{C}_{k-1}(n)q^n. \label{28}
				\end{align}
			Now	addition of equations \eqref{27} and \eqref{28} yields
			$$	N_{0,k}(n)=\frac{\bar{{C}}_{k-1}+{C}_{k-2}(n)}{2},$$
			and subtraction of equations \eqref{27} and \eqref{28} yields
			$$N_{1,k}(n)=\frac{\bar{{C}}_{k-1}-{C}_{k-2}(n)}{2},$$
			which completes the proof.
			\end{proof}
				\begin{proof}[\textbf{Theorem \textup{\ref{O}}}][]
						\begin{align*}
							\sum_{n\ge0}^{}{R}_{3}(n)q^n&=(-q;q^2)_{\infty}^{3}(-q^2;q^2)_{\infty}^{2}\\
							&=(-q;q)^2_\infty(-q;q^2)_\infty\\
							&=\frac{(-q;q^2)_\infty}{(q;q^2)_\infty^2}\\
							&\equiv\frac{1}{(q;q^2)_\infty} \pmod 2\\
							&\equiv(q;q)_\infty\pmod 2.
						\end{align*}
						By using Euler's pentagonal number theorem (\cite[p.~10, Theorem 1.6]{andrewsbook}),
						\begin{align*}
							P_e(d,n)-P_o(d,n)=
							\begin{cases}
								(-1)^m, & \textup{if}\ n=\frac{1}{2}m(3m\pm1) \\
								0, & \text{otherwise},
							\end{cases}
						\end{align*}
						where $P_e(d,n)$ (resp. $P_o(d,n)$) denote the number of partitions
						of $n$ into an even (resp. odd) number of distinct parts, we completes our proof.
					\end{proof}
						\begin{proof}[\textbf{Theorem \textup{\ref{R1}}}][]
							We already prove that,
							\begin{align}\label{10}
										R_{0,3}(n)+R_{1,3}(n)=\bar{O}_2(n). 
							\end{align}
			Now,
						\begin{align}
							\sum_{n=0}^{\infty}\left(R_{0,3}(n)-R_{1,3}(n)\right)q^n&=(-q;q^2)_{\infty}^{2}(-q^2;q^2)_{\infty}^{}(q;q)_\infty\nonumber\\
							&=(-q;q^2)_{\infty}^{}(-q;q^2)_{\infty}^{}(-q^2;q^2)_\infty(q;q)_\infty\nonumber\\
							&={(q^2;q^2)_{\infty}^{}(-q;q^2)_{\infty}}\nonumber\\
							&=f(q;-q^2),\nonumber
						\end{align}
						where $f(q;-q^2)$ is the famous general Ramanujan theta function with $a=q$ and $b=-q^2$ and defined by \cite[Definition 1.2.1, P.~6]{bruce}
						$$f(a,b)=\sum_{n=-\infty}^{\infty}a^{\frac{m(m+1)}{2}}b^{\frac{m(m-1)}{2}},\ \ \ \ |ab|<1.
						$$
						Therefore,
						\begin{align}
								\sum_{n=0}^{\infty}\left(R_{0,3}(n)-R_{1,3}(n)\right)q^n&=\sum_{m=-\infty}^{\infty}(-1)^{\frac{m}{2}(3m+1)}q^{\frac{m}{2}(3m-1)}\nonumber\\
								&=\sum_{m=0}^{\infty}(-1)^{\frac{m}{2}(3m+1)}q^{\frac{m}{2}(3m-1)}+\sum_{-\infty}^{1}(-1)^{\frac{m}{2}(3m+1)}q^{\frac{m}{2}(3m-1)}\nonumber\\
							&=\sum_{m=0}^{\infty}(-1)^{\frac{m}{2}(3m+1)}q^{\frac{m}{2}(3m-1)}+\sum_{m=1}^{\infty}(-1)^{\frac{m}{2}(3m-1)}q^{\frac{m}{2}(3m+1)}\nonumber\\
							&={1+\sum_{m=1}^{\infty} (-1)^{\frac{m(3m\mp1)}{2}q^{\frac{m(3m\pm1)}{2}}}}\nonumber\\
									&={1+\sum_{m=1}^{\infty} (-1)^{\frac{m(3m\mp1)}{2}q^{\frac{m(3m\pm1)}{2}}}}.\label{a}
						\end{align}
						Now addition of \eqref{10} and \eqref{a} yields
							\item $$R_{0,3}(n) = \begin{cases} \frac{\bar{O}_{2}(n) }{2} + \frac{(-1)^{\frac{m(3m\mp1)}{2}}}{2} & \text{if } n=\frac{m(3m\pm1)}{2} \\
							0& \text{otherwise,} \end{cases}$$
							
						and subtraction of \eqref{10} and \eqref{a} yields
							\item $$R_{1,3}(n) = \begin{cases} \frac{\bar{O}_{2}(n) }{2} - \frac{(-1)^{\frac{m(3m\mp1)}{2}}}{2} & \text{if } n=\frac{m(3m\pm1)}{2} \\
							0& \text{otherwise.} \end{cases}$$
					\end{proof}
				\begin{proof}[\textbf{Corollary \textup{\ref{c3}}}][]
				Equation \eqref{a} implies that,
				$$	\sum_{n=0}^{\infty}\left(R_{0,3}(n)-R_{1,3}(n)\right)q^n={1+\sum_{m=1}^{\infty} (-1)^{\frac{m(3m\mp1)}{2}}q^{\frac{m(3m\pm1)}{2}}},$$
				now on compairing coefficient of $q^n$ both sides, we obtain the required result.
				\end{proof}
	\section{Proofs for Generalizations of partition functions of Andrews-Bachraoui}\label{gene}
		\begin{proof}[\textbf{Theorem \textup{\ref{gen}}}][]
			Clearly,
				\begin{align}\label{1}	\sum_{n=0}^{\infty}\left(N_{2,k}(n)-N_{3,k}(n)\right)q^n=\frac{1}{(q;q^2)_{\infty}^{k-1}(q^2;q^2)_{\infty}^{k-1}(-q;q^2)_\infty},
					\end{align}
				and
				\begin{align}\label{2}	\sum_{n=0}^{\infty}\left(N_{2,k}(n)+N_{3,k}(n)\right)q^n=\frac{1}{(q;q^2)_{\infty}^{k-1}(q^2;q^2)_{\infty}^{k-1}(q;q^2)_\infty}.
					\end{align}
				Now, adding \eqref{1} and \eqref{2}, we get
				\begin{align}
					\sum_{n=0}^{\infty}N_{2,k}(n)q^n&=\frac{1}{2(q;q^2)_{\infty}^{k-1}(q^2;q^2)_{\infty}^{k-1}(q^2;q^4)_\infty}\left[(-q;q^2)_\infty+(q;q^2)_\infty\right].\nonumber\\
						\end{align}
							Applying the transformations $x\rightarrow\frac{1}{q}$ and $q\rightarrow q^4$, and subsequently $x\rightarrow-\frac{1}{q}$ with $q\rightarrow q^4$ to \eqref{jacobi}, we obtain
									\begin{align}
			&=\frac{1}{(q;q^2)_{\infty}^{k-1}(q^2;q^2)_{\infty}^{k}}\sum_{n=-\infty}^{\infty}q^{2n^2+n}(1+(-1)^n)\nonumber\\
					&=\frac{1}{(q;q^2)_{\infty}^{k-1}(q^2;q^2)_{\infty}^{k}}\sum_{n=-\infty}^{\infty}q^{8n^2+2n}\nonumber\\
					&=\frac{(q^{16},-q^6,-q^{10};q^{16})_\infty}{(q;q^2)_{\infty}^{k-1}(q^2;q^2)_{\infty}^k}\nonumber,
				\end{align}
				where in last step we applied \eqref{jacobi} with $x\rightarrow\frac{1}{q^6}$ and $q\rightarrow q^{16}$.
	\end{proof}
		\begin{proof}[\textbf{Corollary \textup{\ref{p}}}][]
		This immediately follows by taking $k=2$ in Theorem \ref{gen}.
	\end{proof}
		\begin{proof}[\textbf{Corollary \textup{\ref{nnn}}}][]
		This immediately follows by taking $k=2$ in Theorem \ref{gen}.
	\end{proof}
	\begin{proof}[\textbf{Theorem \textup{\ref{gen2}}}][]
	Subtracting \eqref{1} from \eqref{2}, we obtain
		\begin{align}
			\sum_{n=0}^{\infty}N_{3,k}(n)q^n&=\frac{1}{2(q;q^2)_{\infty}^{k-1}(q^2;q^2)_{\infty}^{k-1}(q^2;q^4)_\infty}\left[(-q;q^2)_\infty-(q;q^2)_\infty\right]\nonumber\\
			&=\frac{1}{(q;q^2)_{\infty}^{k-1}(q^2;q^2)_{\infty}^{k}}\sum_{n=-\infty}^{\infty}q^{2n^2+n}(1-(-1)^n)\nonumber\\
			&=\frac{1}{(q;q^2)_{\infty}^{k-1}(q^2;q^2)_{\infty}^{k}}\sum_{n=-\infty}^{\infty}q^{8n^2+6n+1}\nonumber\\
			&=q\frac{(q^{16},-q^2,-q^{14};q^{16})_\infty}{(q;q^2)_{\infty}^{k-1}(q^2;q^2)_{\infty}^k},\nonumber
		\end{align}
		where in last step we applied \eqref{jacobi} with $x\rightarrow\frac{1}{q^2}$ and $q\rightarrow q^{16}$.
		\end{proof}
			\begin{proof}[\textbf{Corollary \textup{\ref{gc3}}}][]
			This immediately follows by taking $k=2$ in Theorem \ref{gen2}.
		\end{proof}
		\begin{proof}[\textbf{Corollary \textup{\ref{nnnn}}}][]
			This immediately follows by taking $k=2$ in Theorem \ref{gen2}.
		\end{proof}
		\begin{proof}[\textbf{Theorem \textup{\ref{gen3}}}][]
			Observe that,
			\begin{align} 
				\sum_{n \geq 0} (N_{4,k}(n) - N_{5,k}(n)) q^n &= \frac{1}{(-q^2; q^2)_{\infty}(q^2; q^2)_{\infty}^{k-2} (q; q^2)_{\infty}^k},\label{3} \\
				\sum_{n \geq 0} (N_{4,k}(n) + N_{5,k}(n)) q^n &= \frac{1}{(q^2; q^2)_{\infty}^{k-1} (q; q^2)_{\infty}^k}. \label{4}
			\end{align}
			Then by adding \eqref{3} and \eqref{4} we have,
			\begin{align}
				\sum_{n \ge 0} N_{4,k}(n) q^n &= \frac{1}{2} \left( \frac{1}{(q^2; q^2)_{\infty}^{k-1} (q; q^2)_{\infty}^k} + \frac{1}{(-q^2; q^2)_{\infty}(q^2; q^2)_{\infty}^{k-2} (q; q^2)_{\infty}^k} \right) \nonumber \\
				&= \frac{1}{2 (q^2; q^2)_{\infty}^{k-1} (q; q^2)_{\infty}^k} \left( 1 + \frac{(q^2; q^2)_{\infty}}{(-q^2; q^2)_{\infty}} \right). \nonumber\\
				&= \frac{1}{2 (q^2; q^2)_{\infty}^{k-1} (q; q^2)_{\infty}^k} \left( 1 + (q^2; q^4)_{\infty} (q^4; q^4)_{\infty} (q^2; q^4)_{\infty} \right)\nonumber \\
				&= \frac{1}{2 (q^2; q^2)_{\infty}^{k-1} (q; q^2)_{\infty}^k} \left( 1 + (q^4; q^4)_{\infty} (q^2; q^4)_{\infty}^2 \right).\nonumber\\
			\end{align}
			Now with \eqref{jacobi} with $q\rightarrow q^4$ and $x=-q^{-2}$,
			\begin{align}
				\sum_{n \ge 0} N_{4,k}(n) q^n &= \frac{1}{2 (q^2; q^2)_{\infty}^{k-1} (q; q^2)_{\infty}^k} \left( 1 + \sum_{n=-\infty}^{\infty} (-1)^n q^{2n^2} \right) \label {4.4} \\
				&= \frac{1}{(q^2; q^2)_{\infty}^{k-1} (q; q^2)_{\infty}^k} \sum_{n \ge 0} q^{8n^2} (1 - q^{8n+2}) \nonumber \\
				&= \frac{1}{(q^2;q^2)_{\infty}^{k-2}(q; q^2)_{\infty}^{k-1}} \sum_{n \ge 0} \frac{q^{2+6+10+\dots+(8n-2)}}{	 \prod\limits_{\substack{j=1 \\ j \neq 8n+2}}^{\infty} (1 - q^{j})}\nonumber\\
					&=\sum_{n=0}^{\infty}{p}_{4,2}(n,k)q^n. 
			\end{align}
		Subtracting \eqref{3} from \eqref{4}, we find that
				\begin{align}
				\sum_{n \ge 0} N_{5,k}(n) q^n &= \frac{1}{2} \left( \frac{1}{(q^2; q^2)_{\infty}^{k-1} (q; q^2)_{\infty}^k} - \frac{1}{(-q^2; q^2)_{\infty}(q^2; q^2)_{\infty}^{k-2} (q; q^2)_{\infty}^2} \right) \nonumber \\
				&= \frac{1}{2 (q^2; q^2)_{\infty}^{k-1} (q; q^2)_{\infty}^k} \left( 1 -\frac{(q^2; q^2)_{\infty}}{(-q^2; q^2)_{\infty}} \right) \nonumber\\
				&= \frac{1}{2 (q^2; q^2)_{\infty}^{k-1} (q; q^2)_{\infty}^k} \left( 1 -(q^2; q^4)_{\infty} (q^4; q^4)_{\infty} (q^2; q^4)_{\infty} \right)\nonumber \\
				&= \frac{1}{2 (q^2; q^2)_{\infty}^{k-1} (q; q^2)_{\infty}^k} \left( 1 - (q^4; q^4)_{\infty} (q^2; q^4)_{\infty}^2 \right).\nonumber
			\end{align}
			Now with \eqref{jacobi} with $q\rightarrow q^4$ and $x=-q^{-2}$,
			\begin{align}
				\sum_{n \ge 0} N_{5,k}(n) q^n &= \frac{1}{2 (q^2; q^2)_{\infty}^{k-1} (q; q^2)_{\infty}^k} \left( 1 + \sum_{n=-\infty}^{\infty} (-1)^{n+1} q^{2n^2} \right)\nonumber\\
				&= \frac{1}{ (q^2; q^2)_{\infty}^{k-1} (q; q^2)_{\infty}^k} \left(\sum_{n=1}^{\infty} (-1)^{n+1} q^{2n^2} \right)\nonumber\\
				&= \frac{1}{ (q^2; q^2)_{\infty}^{k-1} (q; q^2)_{\infty}^k} \left(\sum_{n=0}^{\infty} (-1)^{n} q^{2n^2+4n+2} \right) \nonumber\\
				&= \frac{1}{(q^2; q^2)_{\infty}^{k-1} (q; q^2)_{\infty}^k} \sum_{n \ge 0} q^{8n^2+8n+2} (1 - q^{8n+6}) \nonumber \\
				&= \frac{1}{(q^2;q^2)_{\infty}^{k-2}(q; q^2)_{\infty}^{k-1}} \sum_{n \ge 0} \frac{q^{2+6+10+\dots+(8n-2)}}{	 \prod\limits_{\substack{j=1 \\ j \neq 8n+6}}^{\infty} (1 - q^{j})}\nonumber\\
				&=\sum_{n=0}^{\infty}\bar{p}_{4,2}(n,k)q^n.
			\end{align}
			This establishes the Theorem \ref{gen3}.
			\end{proof}
				\begin{proof}[\textbf{Corollary \textup{\ref{gc4}}}][]
				This immediately follows by taking $k=2$ in Theorem \ref{gen3}.
			\end{proof}
			\section{Concluding remarks}
			We propose the following questions for further investigation:
			\begin{itemize}
				\item Can we provide a bijective proof for any of the theorem presented in this paper?
				\item In this paper we prove that 
				$$C_{k}^e(n)=\mathcal{P}_{2,1}(n,k),$$
				which can be considered as a natural extension of the theorem provided by Andrews and Newman \cite[Theorem 4]{mex}
					$$\mathcal{P}_{2,1}(n,1)=C_1^e(n)=P_e(n).$$
	It would be great to find partition-theoretic object which is able to extends the following resuls by Andrews and Newman \cite[Theorem 2,~3]{mex},
	\begin{itemize}
		\item [1.]  $\mathcal{P}_{1,1}(n,1)$ equals the number of partitions of $n$ with non-negative crank.
		\item [2.] $\mathcal{P}_{3,3}(n,1)$ equals the number of partitions of $n$ with rank $\ge1.$
	\end{itemize}	 
			\end{itemize}

			\medskip
			
			\noindent {\bf{Acknowledgements:}}\,
			The author is thankful to Professor Rahul Kumar for suggesting this problem and for his various important suggestions throughout the work.
	
\end{document}